\documentclass[11pt]{article}

\usepackage{amsthm,amsmath,amsfonts}

\newtheorem{theorem}{Theorem}
\newtheorem{lemma}{Lemma}
\newtheorem{proposition}{Proposition}
\newtheorem{corollary}{Corollary}
\newtheorem{remark}{Remark}

\begin{document}

\begin{center}
{\Large
\textbf{The Asymptotic Distribution of Randomly Weighted Sums and Self-normalized Sums}}

\vspace{1cm}

\textsc{P\'eter Kevei}
\footnote{Research supported by the TAMOP--4.2.1/B--09/1/KONV--2010--0005 project.} \\
Analysis and Stochastics Research Group of the Hungarian Academy of Sciences \\
Bolyai Institute, Aradi v\'ertan\'uk tere 1, 6720 Szeged, Hungary,  and \\
CIMAT, Callej\'on Jalisco S/N, Mineral de Valenciana, Guanajuato 36240, Mexico. \\
e-mail: \texttt{kevei@math.u-szeged.hu}

\bigskip

\textsc{David M. Mason} \footnote{Research partially supported by NSF
Grant DMS--0503908.} \\

University of Delaware \\
213 Townsend Hall, Newark, DE 19716, USA \\
e-mail: \texttt{davidm@udel.edu}
\end{center}

\bigskip

\begin{abstract}
We consider the self-normalized sums
$T_{n}=\sum_{i=1}^{n}X_{i}Y_{i}/\sum_{i=1}^{n}Y_{i}$,
where $\{ Y_{i} : i\geq 1 \}$ are non-negative i.i.d.~random variables,
and $\{ X_{i} : i\geq 1 \}  $ are i.i.d. random variables, independent of
$\{ Y_{i} : i \geq 1 \}$. The main result of the paper is that each subsequential
limit law of $T_n$ is continuous for any non-degenerate $X_1$ with finite expectation, if and only
if $Y_1$ is in the centered Feller class.

\textit{Keywords:} Self-normalized sums; Feller class; stable distributions.

\textit{AMS Subject Classificiation:} MSC 60F05; MSC 60E07.

\end{abstract}

\section{Introduction}

Let $\left\{  Y,\mbox{ }Y_{i}:i\geq1\right\}  $ denote a sequence of i.i.d.
random variables, where $Y$ is non-negative and non-degenerate with cumulative
distribution function [cdf] $G$. \ Now let $\left\{  X,\mbox{ }X_{i}%
:i\geq1\right\}  $ be a sequence of i.i.d. random variables, independent of
$\left\{  Y,\mbox{ }Y_{i}:i\geq1\right\}  $, where $X$ is in the class
$\mathcal{X}$ of non-degenerate random variables $X$ satisfying $E|X|<\infty.$
Consider the randomly weighted sums and self-normalized sums
\[
W_{n}=\sum_{i=1}^{n}X_{i}Y_{i}\text{ and }T_{n}=\sum_{i=1}^{n}X_{i}Y_{i}%
/\sum_{i=1}^{n}Y_{i}.
\]
We define $0/0:=0$. \smallskip

In statistics $T_{n}$ has uses as a version of the weighted bootstrap, where
typically more assumptions are imposed on $X$ and $Y$. See Mason and Newton
\cite{MN} for details. We shall see that $T_{n}$ is an interesting random
variable, which is worthy of study in its own right. \smallskip

Notice that $E|X|<\infty$ implies that $T_{n}$ is stochastically bounded and
thus every subsequence of $\left\{  n\right\}  $ contains a further
subsequence $\left\{  n^{\prime}\right\}  $ such that for some random variable
$T$, $T_{n^{\prime}}$ $\overset{\mathrm{D}}{\longrightarrow}T\text{. }$
Theorem 4 of Breiman \cite{Brei} says that $T_{n}$ converges in distribution
along the full sequence $\left\{  n\right\}  $ for \textit{every}
$X\in\mathcal{X}$ with at least one limit law being non-degenerate if and only
if%
\begin{equation}
Y\in D\left(  \beta\right)  ,\text{ with }0\leq\beta<1. \label{DB}%
\end{equation}
In this paper, $Y\in D\left(  \beta\right)  $ means that for some function $L$
slowly varying at infinity and $\beta\geq0$,
\[
P\left\{  Y>y\right\}  =y^{-\beta}L(y),\mbox{ }y>0.
\]
In the case $0<\beta<1$ this is equivalent to $Y\geq0$ being in the domain of
attraction of a positive stable law of index $\beta$. Breiman \cite{Brei} has
shown in his Theorem 3 that in this case $T$ has a distribution related to the
arcsine law. We give a natural extension of his result in Theorem
\ref{Breiman-th3} below. \smallskip

At the end of his paper Breiman conjectured that $T_{n}$ converges in
distribution to a non-degenerate law for\textit{\ some} $X\in\mathcal{X}$ if
and only if $Y\in D\left(  \beta\right)  ,$ with $0\leq\beta<1.$ Mason and
Zinn \cite{MZ} partially verified his conjecture. They established the
following: \medskip

Whenever $X$ is non-degenerate and satisfies $E|X|^{p}<\infty$\ for some
$p>2,$\ then $T_{n}$\ converges in distribution to a non-degenerate random
variable if and only if (\ref{DB}) holds.\smallskip

We shall not solve Breiman's full conjecture in this paper. Our interest is to
investigate the asymptotic distributional behavior of the weighted sums
$W_{n}$ and $T_{n}$ along subsequences $\left\{  n^{\prime}\right\}  $ of
$\left\{  n\right\}  $. An important role in our study is played by those $Y$
that are in the \textit{centered Feller class}. A random variable $Y$ (not
necessarily non-negative) is said to be in the \textit{Feller class} if there
exist sequences of norming and centering constants $\{  a_n \}_{n\geq1}$
and $\left\{  b_{n}\right\}  _{n\geq1}$ such that if
$Y_{1},Y_{2},\dots$ are i.i.d. $Y$ then for every subsequence of $\left\{
n\right\}  $ there exists a further subsequence $\left\{  n^{\prime}\right\}
$ such that
\[
\frac{1}{a_{n^{\prime}}}\left\{  \sum_{i=1}^{n^{\prime}}Y_{i}-b_{n^{\prime}%
}\right\}  \overset{\mathrm{D}}{\longrightarrow}W,\text{ as }n^{\prime
}\rightarrow\infty,
\]
where $W$ is a non-degenerate random variable. We shall denote this by
$Y\in\mathcal{F}$. Furthermore, $Y$ is in the \textit{centered Feller class},
if $Y$ is in the \textit{Feller class} and one can choose $b_{n}=0$,
for all $n\geq1$. This we shall denote as $Y\in\mathcal{F}_{c}$.
In this paper the norming sequence $\{a_n \}$ is always
assumed to be strictly positive and to tend to infinity.

\smallskip

Our most unexpected finding is the following theorem, which connects
$Y\in\mathcal{F}_{c}$ with  the continuity of all of the subsequential
limit laws
of $T_{n}$. It is an immediate consequence of the results that we shall establish.

\begin{theorem}
\label{continuous} All subsequential distributional limits of
\[
T_{n}=\frac{\sum_{i=1}^{n}X_{i}Y_{i}}{\sum_{i=1}^{n}Y_{i}}%
\]
are continuous for any $X$ in the class $\mathcal{X}$, if and only if
$Y\in\mathcal{F}_{c}$.$\smallskip$
\end{theorem}

Our result agrees with both Theorem 4 of \cite{Brei} as cited above and
Theorem 3 of \cite{Brei}, which implies that if $Y\in D\left(  \beta\right)
$, with $0<\beta<1,$ then $T_{n}$ $\overset{\mathrm{D}}{\longrightarrow}T$,
where $T$ has a continuous distribution with a Lebesgue density. Note that all
such $Y$ are in the centered Feller class. It turns out that whenever
$Y\in\mathcal{F}_{c}$ and $X \in \mathcal{X}$ every subsequential limit law of $T_{n}$ has a Lebesgue
density. Refer to Theorem \ref{density} below. \smallskip

Breiman \cite{Brei} also studied the randomly weighted sums $W_{n}$. From his
Proposition 3 it can be readily inferred that if $Y\geq0$ and $Y\in D\left(
\beta\right)  $, with $0<\beta<1$, and $X$ is independent of $Y$ satisfying
$E\left\vert X\right\vert <\infty$ then
\begin{align*}
&\lim_{y\rightarrow\infty}\frac{P\{XY>y\}}{1-G(y)}
=\int_{0}^{\infty}x^{\beta}F(\mathrm{d}x) \ \text{ and } \\
&\lim_{y\rightarrow\infty}\frac{P\{XY<-y\}}{1-G(y)}
=\int_{-\infty}^{0}\left(  -x\right)  ^{\beta}F(\mathrm{d}x).
\end{align*}
This implies that for any sequence of norming constants $a_{n}>0$ such that
\begin{equation}
\frac{1}{a_{n}}\sum_{i=1}^{n}Y_{i}\overset{\mathrm{D}}{\longrightarrow
}W\left(  \beta\right)  ,\text{as }n\rightarrow\infty\text{,} \label{st}%
\end{equation}
where $W\left(  \beta\right)  $ is a non-degenerate stable law of index
$\beta$, then for the randomly weighted sums we have
\begin{equation}
\frac{1}{a_{n}}\sum_{i=1}^{n}X_{i}Y_{i}\overset{\mathrm{D}}{\longrightarrow
}W^{\prime}\left(  \beta\right)  ,\text{as }n\rightarrow\infty\text{,}
\label{wst}%
\end{equation}
where $W^{\prime}\left(  \beta\right)  $ is also a non-degenerate stable law
of index $\beta$.\smallskip\ 

Along the way towards establishing the results needed to prove Theorem
\ref{continuous} we shall need to generalize this result. Our Theorem
\ref{2dim-conv} implies that if along a subsequence $\left\{  n^{\prime
}\right\}  $ the normed sum $a_{n^{\prime}}^{-1}\sum_{i=1}^{n^{\prime}}Y_{i}$
converges in distribution then so does $a_{n^{\prime}}^{-1}\sum_{i=1}%
^{n^{\prime}}X_{i}Y_{i}$. It also identifies their limit laws.

Here is a brief outline of our paper. Some necessary notation is introduced in
subsection 1.1, and our main results are stated in subsection 1.2, where we
fill out the picture of the asymptotic distribution of the self-normalized
sums $T_{n}$ along subsequences under a nearly exhaustive set of regularity
conditions. The proofs are detailed in section 2 and some additional
information is provided in an appendix. We shall soon see that the innocuous
looking sequence of stochastic variables $\left\{  T_{n}\right\}  $ displays
quite a variety of subsequential distributional limit behavior.

\subsection{Some necessary notation}

Before we can state our results we must first fix some notation. Let
$\mathrm{id}(a,b,\nu)$ denote an infinitely divisible distribution on
$\mathbb{R}^{d}$ with characteristic exponent
\[
\mathrm{i}u^{\prime}b-\frac{1}{2}u^{\prime}au+\int\left(  e^{\mathrm{i}%
u^{\prime}x}-1-\mathrm{i}u^{\prime}xI \{|x|\leq1 \} \right)  \nu(\mathrm{d}x),
\]
where $b\in\mathbb{R}^{d}$, $a\in\mathbb{R}^{d\times d}$ is a positive
semidefinite matrix and $\nu$ is a L\'{e}vy measure on $\mathbb{R}^{d}$ and
$u^{\prime}$ stands for the transpose of $u$. In our case $d$ is $1$ or $2$.
For any $h>0$ put
\[
a^{h}=a+\int_{|x|\leq h}xx^{\prime}\nu(\mathrm{d}x)\text{ and }b^{h}%
=b-\int_{h<|x|\leq1}x\nu(\mathrm{d}x).
\]
For $d=1$, $\mathrm{id}(\alpha,\Lambda )$, with L\'{e}vy measure $\Lambda$
on $( 0,\infty) $, such that
\begin{equation}
\int_{0}^{1}s\Lambda\left(  \mathrm{d}s\right)  <\infty \label{ff}%
\end{equation}
holds, and $\alpha \geq 0$, denotes a
non-negative infinitely divisible distribution with characteristic exponent
\begin{equation*}
\mathrm{i} u \alpha  + \int_{0}^{\infty } \left( e^{\mathrm{i} u x } - 1 \right)
\Lambda ( \mathrm{d} x ).
\end{equation*}%
Moreover, an infinitely divisible random variable is non-negative, if and only
if the representation above holds.
We will use both representations, so note that
$\mathrm{id}(\alpha, \Lambda) = \mathrm{id}(0,b,\Lambda)$,  if and only if
$\alpha = b - \int_0^1 x \Lambda (\mathrm{d} x)$.

Let $W_{2}$ be an infinitely divisible random variable taking values in
$\left[  0,\infty\right)  $ with characteristic exponent
\begin{equation}
\log Ee^{\mathrm{i}uW_{2}}=\mathrm{i}ub+\int\left(  e^{\mathrm{i}%
ux}-1-\mathrm{i}ux I \{ |x|\leq1 \} \right)  \Lambda(\mathrm{d}x)
=\mathrm{i}u \alpha + \int\left(  e^{\mathrm{i}ux}-1  \right)  \Lambda(\mathrm{d}x),
\label{WW}%
\end{equation}
$b\in\mathbb{R}$ and $\Lambda$ be the L\'{e}vy measure of $W_{2}$ concentrated
on $\left(  0,\infty\right)$ satisfying (\ref{ff}).

Set for $v>0$,
\begin{equation}
\overline{\Lambda}\left(  v\right)  =\Lambda\left(  \left(  v,\infty\right)
\right)  . \label{gam2}%
\end{equation}
We write for $0<v_{1}\leq v_{2}<\infty$%
\[
\int_{v_{1}}^{v_{2}}\Lambda\left(  \mathrm{d}s\right)  =:\int_{\left(
v_{1},v_{2}\right]  }\Lambda\left(  \mathrm{d}s\right)  =\overline{\Lambda
}\left(  v_{1}\right)  -\overline{\Lambda}\left(  v_{2}\right)  =\Lambda
\left(  \left(  v_{1},v_{2}\right]  \right)  .
\]
Note that $\lim_{v_{2}\searrow v_{1}}\Lambda\left(  \left(  v_{1}%
,v_{2}\right]  \right)  =0$ and thus $\overline{\Lambda}\left(  v\right)  $ is
right continuous on $\left(  0,\infty\right)  $; and
\[
\lim_{v_{1}\nearrow v_{2}}\Lambda\left(  \left(  v_{1},v_{2}\right]  \right)
=\Lambda\left(  \left\{  v_{2}\right\}  \right)  .
\]
Let $F$ be the cdf of a random variable $X$ satisfying $0<E|X|<\infty$. We
denote $\overline{F}=1-F.$ For $u\geq0$ and $v>0$ set%
\begin{equation}
\overline{\Pi}\left(  u,v\right)  =\int_{v}^{\infty}\overline{F}\left(
u/s\right)  \Lambda\left(  \mathrm{d}s\right)  =:\int_{\left(  v,\infty
\right)  }\overline{F}\left(  u/s\right)  \Lambda\left(  \mathrm{d}s\right)
\label{v3}%
\end{equation}
and
\begin{equation}
\Pi\left(  -u,v\right)  =\int_{\left(  v,\infty\right)  }F\left(  -u/s\right)
\Lambda\left(  \mathrm{d}s\right)  . \label{v4}%
\end{equation}
In order to define a bivariate L\'{e}vy measure we need to verify that the
functions above are meaningful when $u>0$ and $v=0$. First we shall check
that
\[
\overline{\Pi}\left(  u,0\right)  =\int_{0}^{\infty}\overline{F}%
(u/s)\Lambda(\mathrm{d}s)<\infty,
\]
which is equivalent to the finiteness of $\int_{0}^{1}\overline{F}%
(u/s)\Lambda(\mathrm{d}s)$. Since $E|X|<\infty$, we have $x[F(-x)+\overline
{F}(x)]\rightarrow0$ as $x\rightarrow\infty$, and so by (\ref{ff})
\[
\int_{0}^{1}\overline{F}(u/s)\Lambda(\mathrm{d}s)=\int_{0}^{1}s\,s^{-1}%
\overline{F}(u/s)\Lambda(\mathrm{d}s)\leq u^{-1}\sup_{x\geq0}x\overline
{F}(x)\,\int_{0}^{1}s\Lambda(\mathrm{d}s)<\infty.
\]
The finiteness of (\ref{v4}) with $u>0$ and $v=0$ can be shown in the same way.

Using the functions $\overline{\Pi}(  u,v )$ and $\Pi( -u,v )$
we define the L\'{e}vy measure $\Pi$ on $(-\infty,\infty) \times (0, \infty)$
by
\begin{equation}
\Pi\left(  \left(  a,b\right]  \times\left(  c,d\right]  \right)  =\int
_{c}^{d}\left(  F\left(  b/s\right)  -F\left(  a/s\right)  \right)
\Lambda\left(  \mathrm{d}s\right)  \label{LM}%
\end{equation}
for $-\infty<a<b<\infty$ and $0<c<d<\infty$.

\subsection{Our results}

In this subsection we state our results on the asymptotic distributional
behavior of $W_{n}$ and $T_{n}$ along subsequences $\left\{  n^{\prime
}\right\}  $. Our first theorem is a generalization of the convergence in
distribution fact stated in (\ref{st}) and (\ref{wst}) above. In the
following, $\left\{  (X,Y),\left(  X_{i},Y_{i}\right)  ,i\geq1\right\}  $, are
i.i.d., where $X$ and $Y$ are independent, $X$ has cdf $F$ and $Y$ has cdf $G
$, with $0<P\left\{  Y>0\right\}  \leq P\left\{  Y\geq0\right\}  =1$.

\begin{theorem}
\label{2dim-conv} Assume that $E|X|<\infty$. If along a subsequence $\left\{
n^{\prime}\right\}  $ for a sequence of norming constant $a_{n^{\prime}}>0$
\begin{equation}
\frac{1}{a_{n^{\prime}}}\sum_{i=1}^{n^{\prime}}Y_{i}\overset{\mathrm{D}%
}{\longrightarrow}W_{2},\text{ as }n^{\prime}\rightarrow\infty, \label{Y-conv}%
\end{equation}
where $W_{2}$ has $\mathrm{id}(\alpha, \Lambda) = \mathrm{id}(0,b,\Lambda)$
distribution as in (\ref{WW}) and necessarily
\begin{equation}
\alpha=b-\int_{0}^{1}x\Lambda(\mathrm{d}x)\geq0\text{,} \label{alp1}%
\end{equation}
then along the same subsequence
\begin{equation}
\left(  \frac{\sum_{i=1}^{n^{\prime}}X_{i}Y_{i}}{a_{n^{\prime}}},\frac
{\sum_{i=1}^{n^{\prime}}Y_{i}}{a_{n^{\prime}}}\right)  \overset{\mathrm{D}%
}{\longrightarrow}(W_{1},W_{2}),\text{ as }n^{\prime}\rightarrow\infty,
\label{XY-conv}%
\end{equation}
where $(W_{1},W_{2})$ has $\mathrm{id}(\mathbf{0},\mathbf{b},\Pi)$
distribution, with
\begin{equation}
\mathbf{b}=\left(
\begin{array}
[c]{c}%
b_{1}\\
b_{2}%
\end{array}
\right)  =\left(
\begin{array}
[c]{c}%
\alpha EX+\int_{0<u^{2}+v^{2} \leq1} u \Pi\left(  \mathrm{d}u,\mathrm{d}v\right) \\
\alpha +\int_{0<u^{2}+v^{2}\leq 1}v\Pi\left(  \mathrm{d}u,\mathrm{d}v\right)
\end{array}
\right)  , \label{b}%
\end{equation}
i.e.~it has characteristic function
\begin{equation} \label{limit-chfunc}
\begin{split}
&  Ee^{\mathrm{i}(\theta_{1}W_{1}+\theta_{2}W_{2})}=\exp\bigg\{\mathrm{i}%
(\theta_{1}b_{1}+\theta_{2}b_{2}) \\
+  \int_{0}^{\infty} \! \int_{-\infty}^{\infty} & \left(   e^{\mathrm{i}(\theta
_{1}x+\theta_{2}y)}-1-\left(  \mathrm{i}\theta_{1}x+\mathrm{i}\theta
_{2}y\right) I\left\{ x^{2}+y^{2} \leq 1 \right\} \right)  F\Big(
\frac{\mathrm{d}x}{y} \Big)  \Lambda\left(  \mathrm{d}y\right)  \bigg\}.
\end{split}
\end{equation}
\end{theorem}

\begin{remark}
In general, Theorem \ref{2dim-conv} is no longer valid if $E\left\vert
X\right\vert =\infty$. For example, let $X$ and $Y$ be non-negative,
non-degenerate random variables such that $X\in D(\beta_{1})$ and $Y\in
D(\beta_{2})$, with $0<\beta_{1}<\beta_{2}<1$. We have $EX=\infty.$\ From
Lemma \ref{Breiman-prop3} below we can conclude that $XY$ is in the domain of
attraction of positive stable law of index $\beta_{1}.$ In this example for
sequences of norming constants $a_{n,i}=L_{i}\left(  n\right)  n^{1/\beta_{i}%
},$ $i=1,2$, where $L_{i}\left(  x\right)  ,$ $i=1,2$, are slowly varying
functions at infinity,
\[
a_{n,1}^{-1}\sum_{i=1}^{n}X_{i}Y_{i}\overset{\mathrm{D}}{\longrightarrow}%
W_{1} \ \text{ and } \ a_{n,2}^{-1}\sum_{i=1}^{n}Y_{i}\overset{\mathrm{D}%
}{\longrightarrow}W_{2}\text{, as }n\rightarrow\infty\text{,}%
\]
where $W_{i}$ are non-degenerate stable random variables of index $\beta_{i},
$ $i=1,2$. Since $a_{n,1}/a_{n,2}\rightarrow\infty$, (\ref{XY-conv}) cannot
hold. It is clear in this example that the self-normalized sum $T_{n}$
$\overset{\mathrm{P}}{\longrightarrow}\infty$, which says that $T_{n}$ is not
stochastically bounded.
\end{remark}

\begin{remark}
Note that
\begin{equation}
(W_{1},W_{2})\overset{\mathrm{D}}{=}(a_{1}+U,a_{2}+V), \label{w1w2}%
\end{equation}
where $(a_{1},a_{2})=\left( \alpha EX, \alpha \right)  $ and
\begin{equation}
Ee^{\mathrm{i}(\theta_{1}U+\theta_{2}V)}
= \exp\left\{  \int_{0}^{\infty}%
\int_{-\infty}^{\infty}\left(  e^{\mathrm{i}(\theta_{1}x+\theta_{2}%
y)}-1\right)  F\left(  \mathrm{d}x/y\right)  \Lambda\left(  \mathrm{d}%
y\right)  \right\} =: \exp\left\{  \phi\left(  \theta_{1},\theta_{2}\right)
\right\}  . \label{uv}%
\end{equation}
Furthermore under the assumptions of Theorem \ref{2dim-conv}, we have that the
convergence takes place in the Skorohod space $D(\mathbb{R}_{+},\mathbb{R}%
^{2})$, i.e.
\[
\left\{  \left(  \frac{\sum_{1\leq i\leq n^{\prime}t}X_{i}Y_{i}}{a_{n^{\prime
}}},\frac{\sum_{1\leq i\leq n^{\prime}t}Y_{i}}{a_{n^{\prime}}}\right)
,t>0\right\}  \overset{\mathrm{D}}{\longrightarrow}\left\{  (a_{1}%
t+U_{t},a_{2}t+V_{t}),t>0\right\}  ,
\]
as $n^{\prime}\rightarrow\infty$,
where $(U_{t},V_{t})$, $t\geq0$, is the bivariate L\'{e}vy process with
characteristic function%
\begin{equation}
Ee^{\mathrm{i}(\theta_{1}U_{t}+\theta_{2}V_{t})}=:\exp\left\{  t\phi\left(
\theta_{1},\theta_{2}\right)  \right\}  . \label{LPt}%
\end{equation}
This immediately follows from Theorem \ref{2dim-conv} combined with Skorohod's
theorem (Theorem 16.14 in \cite{Kallenberg}).
\end{remark}

In a separate paper we shall characterize when under regularity conditions
the ratio $ U_{t} / V_{t}$
converges in distribution to a non-degenerate random variable $T$ as
$t\rightarrow\infty$ or $t\searrow0$.
\smallskip

\begin{remark}
A result closely related to Theorem \ref{2dim-conv} is the fact that the
Feller class $\mathcal{F}$ is closed under independent multiplication. It is
established in Proposition \ref{Feller-class} \ in the Appendix that if $X$
and $Y$ are independent random variables in the Feller class, then so is $XY$.
\end{remark}

\begin{remark}
Suppose $E\left\vert X\right\vert <\infty$ and assume that along a subsequence
$\left\{  n^{\prime}\right\}  $ of $\left\{  n\right\}  $ for some sequence
$c_{n^{\prime}}\rightarrow\infty$,
\begin{equation}
\frac{1}{c_{n^{\prime}}}\sum_{i=1}^{n^{\prime}}Y_{i}\overset{\mathrm{P}%
}{\longrightarrow}1,\text{ as }n^{\prime}\rightarrow\infty. \label{prob}%
\end{equation}
By applying Theorem \ref{2dim-conv} we see then that
\begin{equation}
\frac{1}{c_{n^{\prime}}}\sum_{i=1}^{n^{\prime}}X_{i}Y_{i}\overset{\mathrm{P}%
}{\longrightarrow}EX,\text{ as }n^{\prime}\rightarrow\infty, \label{BB}%
\end{equation}
which in combination with (\ref{prob}) implies that
\begin{equation}
T_{n'}\overset
{\mathrm{P}}{\longrightarrow}EX,\text{ as }n^{\prime}\rightarrow\infty.
\label{ExX}%
\end{equation}
Notice that (\ref{prob}) holds for the entire sequence $\left\{  n\right\}  $
with $c_{n}=nEY$ when $EY<\infty.$ It is also satisfied whenever along a
subsequence $\left\{  n^{\prime}\right\}  $ for some sequence $b_{n^{\prime}%
}\rightarrow\infty$,
\begin{equation}
\frac{1}{a_{n^{\prime}}}\left\{  \sum_{i=1}^{n^{\prime}}Y_{i}-b_{n^{\prime}%
}\right\}  \overset{\mathrm{D}}{\longrightarrow}W,\text{ as }n^{\prime
}\rightarrow\infty, \label{W}%
\end{equation}
where $W$ is non-degenerate and $b_{n^{\prime}}/a_{n^{\prime}}\rightarrow
\infty,$ as $n^{\prime}\rightarrow\infty.$ A random variable $Y$ that is in
the Feller class but not in the centered Feller class has this property. In
this case (\ref{prob})\ holds with $c_{n^{\prime}}=b_{n^{\prime}}.$
\end{remark}

The following theorem, describes what happens when $Y$ is in the centered
Feller class.

\begin{theorem}
\label{density} Assume $X \in \mathcal{X}$ and $Y\in
\mathcal{F}_{c}$, then for a suitable sequence of norming constants $a_{n}>0$
any subsequence of $\left\{ n\right\}$  contains a further subsequence  $\left\{n^{\prime}\right\}$
such that
\begin{equation}
\left(  \frac{W_{1,n^{\prime}}}{a_{n'}},\frac{W_{2,n'}}{a_{n'}}\right)  :=\left(
\frac{\sum_{i=1}^{n'}X_{i}Y_{i}}{a_{n'}},\frac{\sum_{i=1}^{n'}Y_{i}}{a_{n'}%
}\right)  \text{,} \label{WA}%
\end{equation}
converges in distribution to a non-degenerate random vector, say $\left(  W_{1}%
,W_{2}\right)$, having a $C^{\infty}$ Lebesgue density $f$ on $\mathbb{R}^{2}$,
which implies that the asymptotic distribution of the corresponding ratio
along the subsequence $\left\{  n^{\prime}\right\}  $ satisfies
\begin{equation}
T_{n'}=\frac{W_{1,n'}}{W_{2,n'}}
\overset{\mathrm{D}}{\longrightarrow}\frac{W_{1}}{W_{2}}=:T \label{TT}%
\end{equation}
and has a Lebesgue density $f_{T}$ on $\mathbb{R}$.
\end{theorem}

Corollary \ref{converse1} below is a kind of a converse of this fact.\smallskip

It is known (and easy calculation shows) that if $Y \in D(\beta)$, $\beta \in (0,1)$,
then the non-negative constant $\alpha$ appearing in the representation of the
stable limit law $\mathrm{id}(\alpha, \Lambda)$ is necessarily 0.
(Breiman tacitly uses this fact in the course of his proof of Theorem 3 \cite{Brei}.)
It turns out  that this is true in a far more general setup.

\begin{proposition} \label{alpha0}
Whenever $Y\in\mathcal{F}_{c}$  and non-negative and $a_{n}>0$ is as in (\ref{WA}), every subsequential limit law $V$ 
of $a_{n}^{-1}\sum_{i=1}^{n}Y_{i}$ is of the form $\mathrm{id}(0,\Lambda)$,
i.e.~ $V$ has characteristic function
\begin{equation*}
Ee^{\mathrm{i}uV}=\exp\left\{  \int_{0}^{\infty}\left(  e^{\mathrm{i}%
uy}-1\right)  \Lambda\left(  \mathrm{d}y\right)  \right\}  ,\label{vvv}%
\end{equation*}
with $\Lambda$ being a L\'evy measure concentrated on
$\left(  0,\infty\right)$ satisfying (\ref{ff}).
\end{proposition}

In order to state our next theorem we shall need the following notation. Let
\[
Y_{n,n}=\max\{Y_{1},\ldots,Y_{n}\}=Y_{m(n)},
\]
where to be specific, $m(n)$ is the smallest $1\leq m\leq n$ such that
$Y_{n,n}=Y_{m(n)}$. For any $0<\varepsilon<1$ put
\[
A_{n}\left(  \varepsilon\right)  =\left\{
Y_{m(n)}/\sum_{i=1}^{n}Y_{i} >1-\varepsilon \right\}.
\]
Set
\[
\Delta_{n}=\left\vert T_{n}-X_{m(n)}\right\vert .
\]

\begin{theorem}
\label{converse} Assume that $E |X|<\infty$ and there exists a subsequence
$\{n^{\prime}\}$ such that
\begin{equation}
\lim_{\varepsilon\rightarrow0}\liminf_{n^{\prime}\rightarrow\infty} P \left\{
A_{n^{\prime}}\left(  \varepsilon\right)  \right\}  =:\delta>0, \label{dd}%
\end{equation}
then
\begin{equation}
\lim_{\varepsilon\rightarrow0}\liminf_{n^{\prime}\rightarrow\infty} P \left\{
\Delta_{n^{\prime}}\leq\varepsilon\right\}  \geq\delta>0. \label{BigD}%
\end{equation}
\end{theorem}

In Proposition 1 in \cite{Mason05} Mason proves that whenever $Y$ is not in
the Feller class, that is,
\begin{equation}
\limsup_{x\rightarrow\infty}\frac{x^{2}P\{Y>x\}}{EY^{2}I(Y\leq x)}=\infty,
\label{not-feller-class}%
\end{equation}
and, in addition,
\begin{equation}
\limsup_{x\rightarrow\infty}\frac{xE\left(  YI(Y\leq x)\right)  }%
{x^{2}P\left\{  Y>x\right\}  +EY^{2}I(Y\leq x)}<\infty\label{grif}%
\end{equation}
then there is a subsequence $\{n^{\prime}\}$, such that (\ref{dd}) holds.

Condition (\ref{grif}) is equivalent to
\begin{equation}
\frac{\sum_{i=1}^{n}Y_{i}}{\sqrt{\sum_{i=1}^{n}Y_{i}^{2}}}=O_{P}\left(
1\right)  . \label{op}%
\end{equation}
Consult Griffin \cite{Grif} for more details. \smallskip

Theorem \ref{converse} leads to the following corollary.

\begin{corollary}
\label{converse1} Assume $E|X|<\infty$, (\ref{dd}), and $P\{X=x_{0}\}>0$ for
some $x_{0}$. Then there exists a subsequence $\{n^{\prime}\}$ such that
\begin{equation}
\lim_{\varepsilon\rightarrow0}\liminf_{n^{\prime}\rightarrow\infty}P\left\{
T_{n'}\in(x_{0}-\varepsilon,x_{0}+\varepsilon)\right\}  >0. \label{atom}%
\end{equation}
By the stochastic boundedness of $T_{n}$ this implies that there is a
subsequence $\left\{  n^{\prime}\right\}  $ such that
\[
T_{n^{\prime}}\overset{\mathrm{D}}{\longrightarrow}T,
\]
where $P\{T=x_{0}\}>0$.
\end{corollary}

It is well-known (cf.~Theorem 3.2 by Darling \cite{Dar}) that if $Y$ has a
slowly varying upper tail, which by an application of Theorem 1.2.1 of de Haan
\cite{dehaan} is seen to be equivalent to
\begin{equation}
\lim_{x\rightarrow\infty}\frac{x^{2}P\{Y>x\}}{EY^{2}I(Y\leq x)}=\infty
,\label{SV}%
\end{equation}
then (\ref{dd}) holds along the full sequence $\{n\}$ with $\delta=1$. In this
case (\ref{grif}) holds since (\ref{SV}) implies
\[
\sum_{i=1}^{n}Y_{i}/\sqrt{\sum_{i=1}^{n}Y_{i}^{2}}\overset{\mathrm{P}%
}{\longrightarrow}1.
\]
This leads immediately to Proposition 2 in \cite{Brei}:

\begin{corollary}
\label{converse2} Assume that $E|X|<\infty$ and (\ref{SV}) holds. Then
\begin{equation}
T_{n}\overset{\mathrm{D}%
}{\longrightarrow}X. \label{X}%
\end{equation}
\end{corollary}

Next in the case when $Y$ does not satisfy condition (\ref{grif}) we have the following.

\begin{theorem}
\label{exp} Assume that $E|X|<\infty$ and condition (\ref{grif}) does not hold, then there exists a
subsequence $\left\{  n^{\prime}\right\}  $ of $\left\{  n\right\}  $ and a
random variable $T$ such that
\[
T_{n^{\prime}}\overset{\mathrm{D}}{\longrightarrow}T,
\]
where $P\left\{  T=EX\right\}  >0.$
\end{theorem}

\begin{remark}
Condition (\ref{grif}) (equivalently (\ref{op})) does not hold when $EY<\infty$. To verify this, note
that
\[
R_{n}=\frac{\sqrt{\sum_{i=1}^{n}Y_{i}^{2}}}{\sum_{i=1}^{n}Y_{i}}\leq
\frac{\sqrt{\sum_{i=1}^{n}Y_{i}/n\max\left\{  Y_{1},\dots,Y_{n}\right\}  /n}%
}{\sum_{i=1}^{n}Y_{i}/n}.
\]
Since $EY<\infty$ implies that $\max\left\{  Y_{1},\dots,Y_{n}\right\}
/n\rightarrow0$, a.s., we conclude by the law of large numbers that
$R_{n}\rightarrow0$, a.s. In this case, it is trivial to see that
$T_{n}\rightarrow EX$, a.s., as $n\rightarrow\infty$.$\medskip$
\end{remark}

Finally, let us consider an illustrative case when $E|X|$ is not necessarily
finite. We shall need the following lemma, which is a simple extension of
Breiman's Proposition 3 \cite{Brei}. Since the proof is nearly the same, we
omit it.

\begin{lemma}
\label{Breiman-prop3} Assume that $Y\in D(\beta)$ for some $\beta>0$, and
there exists $\varepsilon>0$ such that $E|X|^{\beta+\varepsilon}<\infty$.
Then
\begin{align*}
\lim_{y\rightarrow\infty}\frac{P\{XY>y\}}{1-G(y)}  &  =\int_{0}^{\infty
}x^{\beta}F(\mathrm{d}x),\\
\lim_{y\rightarrow\infty}\frac{P\{XY<-y\}}{1-G(y)}  &  =\int_{-\infty}%
^{0}(-x)^{\beta}F(\mathrm{d}x).
\end{align*}
\end{lemma}

A more general result is given in Proposition II in Cline \cite{Cline}. For
recent results along this line consult Jessen and Mikosch \cite{Jessen} and
Denisov and Zwart \cite{Den}. \smallskip

By substituting the use of Breiman's Proposition 3 in the proof of his Theorem
3 in \cite{Brei} by the above Lemma \ref{Breiman-prop3}, we obtain the
following extension of his Theorem 3, which implies that his asymptotic
distribution result for $T_{n}$ holds in cases when $E|X|=\infty$.

\begin{theorem}
\label{Breiman-th3} Assume that $Y\in D(\beta)$ for some $\beta\in(0,1)$, and
there exists $\varepsilon>0$ such that $E|X|^{\beta+\varepsilon}<\infty$. Then
$T_{n}\overset{\mathrm{D}}{\rightarrow}T,$ where
\begin{equation}
P\left\{  T\leq x\right\}  =\frac{1}{2}+\frac{1}{\pi\beta}\arctan\left[
\frac{\int|u-x|^{\beta}\mathrm{sgn}(x-u)F(\mathrm{d}u)}{\int|u-x|^{\beta
}F(\mathrm{d}u)}\tan\frac{\pi\beta}{2}\right]  . \label{integ}%
\end{equation}
\end{theorem}

It is interesting that even in the latter case the tail behavior of the limit
distribution is determined by the distribution of $X$. Note that
\[
\lim_{x\rightarrow\pm\infty}\frac{\int|u-x|^{\beta}\mathrm{sgn}%
(x-u)F(\mathrm{d}u)}{\int|u-x|^{\beta}F(\mathrm{d}u)}=\pm1.
\]
Using that as $y\rightarrow0$
\[
\arctan\left(  (1-y)\tan\frac{\pi\beta}{2}\right)  =\frac{\pi\beta}{2}%
-y\tan\frac{\pi\beta}{2}\left(  1+\tan^{2}\frac{\pi\beta}{2}\right)
^{-1}+O(y^{2}),
\]
we obtain then that
\[
P\{T>x\}\sim2\int_{x}^{\infty}\left(  \frac{u}{x}-1\right)  ^{\beta
}F(\mathrm{d}u)\frac{\tan\frac{\pi\beta}{2}}{\pi\beta\left(  1+\tan^{2}%
\frac{\pi\beta}{2}\right)  },\ \text{as }x\rightarrow\infty.
\]

Without any further assumptions on $F$ we have the simple bounds
\[
\int_{x}^{\infty}\left(  \frac{u}{x} - 1 \right)  ^{\beta} F(\mathrm{d} u)
\geq\int_{2x}^{\infty}1 F(\mathrm{d} u) = 1 - F(2 x),
\]
and
\begin{equation*}
\int_{x}^{\infty}\left(  \frac{u}{x} - 1 \right)  ^{\beta} F(\mathrm{d} u)
\leq \int_{x}^{\infty}\left(  \frac{u}{x} \right)  ^{\beta} F(\mathrm{d} u)
= 1 - F(x) + \beta x^{-\beta} \int_{x}^{\infty}
[ 1 - F(u)] u^{\beta- 1} \mathrm{d} u .
\end{equation*}
Moreover, assuming that $1 - F$ is regularly varying with index $- \alpha$,
with $\alpha> \beta$ it is easy to show that
\[
\int_{x}^{\infty}\left(  \frac{u}{x} - 1 \right)  ^{\beta} F(\mathrm{d} u)
\sim(1 - F(x) ) \beta\int_{1}^{\infty}y^{-\alpha} ( y - 1)^{\beta- 1}
\mathrm{d} y,
\]
as $x \to\infty$, i.e.
\[
\lim_{x \to\infty} \frac{P \{ T > x \} }{1 - F(x) } = 2 \beta\int_{1}^{\infty
}y^{-\alpha} ( y - 1)^{\beta- 1} \mathrm{d} y \frac{\tan\frac{\pi\beta}{2}} {
\pi\beta\left(  1 + \tan^{2} \frac{\pi\beta}{2} \right)  }.
\]
Clearly analogous results are true for the negative tail.

The tail behavior that we just pointed out is in sharp contrast to the
classical self-normalized sum setup, where it is shown by Gin\'{e}, G\"{o}tze and
Mason (Theorem 2.5 in \cite{GGM}) that if the ratio $\sum_{i=1}^{n}Y_{i}/\sqrt
{\sum_{i=1}^{n}Y_{i}^{2}}$ is stochastically bounded, then all the
subsequential limits are subgaussian.

\paragraph{Summary picture}

To summarize, we have developed the following picture: Let $X$ and $Y$ be
independent such that $0<P\left\{  Y>0\right\}  \leq P\left\{  Y\geq0\right\}
=1 $.\smallskip

\noindent(i) If $X$ is non-degenerate, $0<E|X|<\infty$ and $Y\in\mathcal{F}_{c}$ then $T_{n}$ is
stochastically bounded and every subsequential limit random variable $T$ has a
Lebesgue density.

\noindent(ii) If $E\left\vert X\right\vert <\infty$ and $Y\in\mathcal{F}$ but
$Y\notin\mathcal{F}_{c}$ then there exists a subsequence $\left\{  n^{\prime
}\right\}  $ such that $T_{n^{\prime}}\overset{\mathrm{P}}{\longrightarrow}%
EX$.

\noindent(iii) The last result is a special case of the fact that if
$E\left\vert X\right\vert <\infty$ and along a subsequence $\left\{
n^{\prime}\right\}  $ and some sequence $c_{n^{\prime}}\rightarrow\infty$, we
have $c_{n^{\prime}}^{-1}\sum_{i=1}^{n^{\prime}}Y_{i}\overset{\mathrm{P}%
}{\longrightarrow}1,$ as $n^{\prime}\rightarrow\infty,$ then $T_{n^{\prime}%
}\overset{\mathrm{P}}{\longrightarrow}$ $EX.$ 

\noindent(iv) If $E\left\vert X\right\vert <\infty$ and $Y\notin\mathcal{F}$
and (\ref{grif}) holds then there exists a subsequence $\left\{  n^{\prime
}\right\}  $ such that for some $\delta>0$
\[
\lim_{\varepsilon\rightarrow0}\liminf_{n^{\prime}\rightarrow\infty}P\left\{
\min_{1\leq i\leq n^{\prime}}\left\vert T_{n^{\prime}}-X_{i}\right\vert
\leq\varepsilon\right\}  \geq\delta.
\]
Moreover, if $Y$ has a slowly varying upper tail%
\[
\lim_{\varepsilon\rightarrow0}\liminf_{n\rightarrow\infty}P\left\{
\min_{1\leq i\leq n}\left\vert T_{n}-X_{i}\right\vert \leq\varepsilon\right\}
=1.
\]

\noindent(v) If $E\left\vert X\right\vert <\infty$, $Y\notin\mathcal{F}$,
(\ref{grif}) holds and $P\{X=x_{0}\}>0$ for some $x_{0}\in\mathbb{R}$, then
there exists a subsequence $\left\{  n^{\prime}\right\}  $ and a random
variable $T$ such that $T_{n^{\prime}}\overset{\mathrm{D}}{\longrightarrow}T$
and $P\{T=x_{0}\}>0$.

\noindent(vi) If $E\left\vert X\right\vert <\infty$ and (\ref{grif}) does not
hold then there exists a subsequence $\left\{  n^{\prime}\right\}  $ and a
random variable $T$ such that $T_{n^{\prime}}\overset{\mathrm{D}%
}{\longrightarrow} T$ and $P\{T=EX\}>0$.

\noindent(vii) It can happen that $E\left\vert X\right\vert =\infty$ and
$Y\in\mathcal{F}_{c}$ and $T_{n}\overset{\mathrm{P}}{\longrightarrow}$
$\infty.$ 

\noindent(viii) On the other hand, it can also happen that $E\left\vert
X\right\vert =\infty$ and $Y\in\mathcal{F}_{c}$ and $T_{n}\overset{\mathrm{D}%
}{\longrightarrow} T$, where $T$ is non-degenerate.

\section{Proofs}

We shall need the following additional notation. Write for $v>0$
\begin{equation}
\overline{\Lambda}_{n}\left(  v\right)  =n P\left\{  Y>a_{n }v\right\}  =n
\overline{G}\left(  a_{n }v\right)  \label{gam1}%
\end{equation}
and for $u>0$ and $v>0$%
\begin{equation}
\overline{\Pi}_{n}\left(  u,v\right)  =n P\left\{  XY>a_{n}u,Y>a_{n}v\right\}
=\int_{v}^{\infty}\overline{F}\left(  u/s\right)  n G\left(  \mathrm{d}%
sa_{n}\right)  \label{v1}%
\end{equation}
and
\begin{equation}
\Pi_{n}\left(  -u,v\right)  =n P\left\{  XY \leq-a_{n} u,Y>a_{n}v\right\}
=\int_{v}^{\infty}F\left(  -u/s\right)  n G\left(  \mathrm{d}sa_{n}\right)  .
\label{v2}%
\end{equation}
The following lemma is well-known (see Corollary 15.16 of (\cite{Kallenberg})):

\begin{lemma}
\label{CHT-Kallenberg} Let $\{\xi_{n,j}\}_{j=1}^{m_{n}}$ be an i.i.d. array in
$\mathbb{R}^{d}$. Then $\sum_{j=1}^{m_{n}}\xi_{n,j}$ converges in distribution
to an infinitely divisible $\mathrm{id}(a,b,\nu)$ random vector if and only if for some (any) $h>0$
with $\nu(x:|x|=h)=0$ we have, with $\overset{v}{\rightarrow}$ denoting vague convergence,

\begin{itemize}
\item[(e.i)] $m_{n}P\circ\xi_{n,1}^{-1}\overset{v}{\rightarrow}\nu$ on
$\mathbb{R}^{d}\backslash\{\mathbf{0}\}$,

\item[(e.ii)] $m_{n} E[\xi_{n,1} I\{ |\xi_{n,1}|\leq h \} ]\rightarrow b^{h}$,

\item[(e.iii)] $m_{n}E[\xi_{n,1}\xi_{n,1}^{\prime}
I\{ |\xi_{n,1}| \leq h \} ]\rightarrow a^{h}$,
\end{itemize}
where $a^h$ and $b^h$ are defined above (\ref{ff}).
\end{lemma}

The following lemma determines the continuity points of the two-di\-men\-sional
L\'evy measure.

\begin{lemma}
\label{levy-cont} Any $\left(  u,v\right)  \in\left[  0,\infty\right)
\times\left(  0,\infty\right)  $ is a continuity point of $\overline{\Pi}$
only if $F\left(  u/s\right)  $ and $\overline{\Lambda}\left(  s\right)  $
as functions of $s$ are not discontinuous at the same points in
$(v,\infty)$ and $\overline{F}\left(  u/v-\right)  \Lambda(\{v\})=0$; and any
$\left(  -u,v\right)  \in\left(  -\infty,0\right]  \times\left(
0,\infty\right)  $ is a continuity point of $\Pi$ only if $F\left(
-u/s\right)  $ and $\overline{\Lambda}\left(  s\right)  $
as functions of $s$ are not discontinuous at the same points in $(v,\infty)$ and
$\overline{F}\left(  -u/v-\right)  \Lambda(\{v\})=0$.
\end{lemma}

\begin{proof}
We see that
\begin{align*}
\lim_{\widetilde{u}\uparrow u,\widetilde{v}\uparrow v}    \left(
\overline{\Pi}\left(  \widetilde{u},\widetilde{v}\right)  -\overline{\Pi
}\left(  u,v\right)  \right)
&  =\lim_{\widetilde{u}\uparrow u,\widetilde{v}\uparrow v}\int_{\widetilde{v}%
}^{v}\overline{F}(\widetilde{u}/s)\Lambda(\mathrm{d}s)+\lim_{\widetilde
{u}\uparrow u}\int_{v}^{\infty}\left(  F\left(  u/s\right)  -F\left(
\widetilde{u}/s\right)  \right)  \Lambda\left(  \mathrm{d}s\right) \\
&  =\overline{F}\left(  u/v-\right)  \Lambda(\{v\})+\int_{v}^{\infty}\left(
F\left(  u/s\right)  -F\left(  u/s-\right)  \right)  \Lambda\left(
\mathrm{d}s\right)  ,
\end{align*}
\textit{\smallskip}which is zero only if $F\left(  u/s\right)  $ and
$\overline{\Lambda}\left(  s\right)  $ are not discontinuous at the same
points in $(v,\infty)$ and $\overline{F}\left(  u/v-\right)  \Lambda
(\{v\})=0$. The second part of the lemma is proved in the same way.
\end{proof}

Next we deal with the convergence of the L\'evy measures.

\begin{proposition} \label{Levy-measure}
Assume that at every continuity point $v \in \left(  0,\infty\right)$
of $\overline{\Lambda}$
\begin{equation}
\overline{\Lambda}_{n^{\prime}}\left(  v\right)  \rightarrow\overline{\Lambda
}\left(  v\right)  \text{, as }n^{\prime}\rightarrow\infty, \label{Gz}%
\end{equation}
and assume that for every (some) continuity point $h>0$ of $\Lambda$
\begin{equation}
\int_{0}^{h}vn^{\prime}G\left(  \mathrm{d}a_{n^{\prime}}v\right)  =\int
_{0}^{h}v{\Lambda}_{n^{\prime}}\left(  \mathrm{d}v\right)  \rightarrow
\alpha_{h}\text{, as }n^{\prime}\rightarrow\infty\text{,} \label{Ga}%
\end{equation}
holds where $\alpha_{h}<\infty$. Then at every continuity
point $\left(  u,v\right)  \in\left[  0,\infty\right)  \times\left[
0,\infty\right)$ of $\overline{\Pi}$ such that $\left(  u,v\right)
\neq\left(  0,0\right)  $
\begin{equation}
\overline{\Pi}_{n^{\prime}}\left(  u,v\right)  \rightarrow\overline{\Pi
}\left(  u,v\right)  \text{, \textit{as} }n^{\prime}\rightarrow\infty,
\label{L1}%
\end{equation}
and at every continuity point $\left(  -u,v\right)  \in\left(
-\infty,0\right]  \times\left[  0,\infty\right)$ of $\Pi$ such that
$\left(  u,v\right)  \neq\left(  0,0\right)$
\begin{equation}
\Pi_{n^{\prime}}\left(  -u,v\right)  \rightarrow\Pi\left(  -u,v\right)
\text{, \textit{as} }n^{\prime}\rightarrow\infty. \label{L2}%
\end{equation}
\end{proposition}

\begin{proof}
First choose any continuity point $(u,v)\in\left[
0,\infty\right)  \times\left(  0,\infty\right)  $ of $\overline{\Pi}$ and let
$\gamma>v$ be a continuity point of $\overline{\Lambda}$. By (\ref{Gz})
\begin{equation}
\limsup_{n^{\prime}\rightarrow\infty}\int_{\gamma}^{\infty}\overline{F}\left(
u/s\right)  \Lambda_{n^{\prime}}\left(  \mathrm{d}s\right)  \leq
\overline{\Lambda}\left(  \gamma\right)  . \label{L4}%
\end{equation}
By Lemma \ref{levy-cont}, $F\left(  u/s\right)  $ and $\overline{\Lambda
}\left(  s\right)  $ are not discontinuous at the same points in $(v,\gamma]$,
and since the set of discontinuities of $F\left(  u/s\right)  $ on $(v,\gamma]
$ is countable and those have $\Lambda$ measure zero, assumption (\ref{Gz})
allows us to conclude that
\begin{equation}
\lim_{n^{\prime}\rightarrow\infty}\int_{v}^{\gamma}\overline{F}\left(
u/s\right)  \Lambda_{n^{\prime}}\left(  \mathrm{d}s\right)  =\int_{v}^{\gamma
}\overline{F}\left(  u/s\right)  \Lambda\left(  \mathrm{d}s\right)  ,
\label{L3}%
\end{equation}
(see the proof of Proposition 8.12 on page 163 of \cite{Breiman-Prob}). Since
$\overline{\Lambda}\left(  \gamma\right)  $ can be made arbitrarily small by
choosing $\gamma$ arbitrarily large we readily infer (\ref{L1}) from
(\ref{L3}) and (\ref{L4}).

To prove the convergence in (\ref{L1}) when $u>0$ and $v=0$ we shall need
assumption (\ref{Ga}). We have to show that for any continuity point
$\gamma>0$
\[
\int_{0}^{\gamma}\overline{F}(u/s)\Lambda_{n^{\prime}}(\mathrm{d}%
s)\rightarrow\int_{0}^{\gamma}\overline{F}(u/s)\Lambda(\mathrm{d}s).
\]
Using that the convergence (\ref{L3}) holds for any continuity points
$0<v<\gamma$ of $\overline{\Lambda}$ it is enough to prove the convergence
\[
\limsup_{n^{\prime}\rightarrow\infty}\int_{0}^{v}\overline{F}(u/s)\Lambda
_{n^{\prime}}(\mathrm{d}s)\rightarrow0,\quad\text{as }v\rightarrow0.
\]
Since $s^{-1}\overline{F}(u/s)\rightarrow0$, as $s\rightarrow0$, (\ref{Ga})
implies the statement keeping mind that $\alpha_{h}\searrow\alpha<\infty$ as
$h\searrow0$ for some finite $\alpha\geq0$. Statement (\ref{L2}) is proved in
the same way.
\end{proof}

\begin{lemma}
\label{F-cont} Put
\begin{equation}
\varphi\left(  v\right)  =\sqrt{h^{2}-v^{2}}/v. \label{phi}%
\end{equation}
For any $v\in\left(  0,h\right]  $, $\left\{  \left(  \sqrt{h^{2}-v^{2}%
},v\right)  \right\}  $ has $\Pi$ measure zero only if $v$ is a continuity
point of $F\left(  \varphi\left(  v\right)  \right)  $ considered as a
function on $\left(  0,h\right]  $, or $\Lambda\left(  \left\{  v\right\}
\right)  =0$; and $\left\{  \left(  -\sqrt{h^{2}-v^{2}},v\right)  \right\}  $
has $\Pi$ measure zero only if $v$ is a continuity point of $F\left(
-\varphi\left(  v\right)  \right)  $, considered as a function on $\left(
0,h\right]  $, or $\Lambda\left(  \left\{  v\right\}  \right)  =0$.
\end{lemma}

\begin{proof}
Select any $0<v<h,$ then for all $v<\widetilde{v}<h$,
we have%
\[
\frac{v\sqrt{h^{2}-\widetilde{v}^{2}}}{\widetilde{v}}<\sqrt{h^{2}%
-\widetilde{v}^{2}}<\sqrt{h^{2}-v^{2}}%
\]
and by (\ref{LM})
\begin{equation*}
\Pi \Bigg(  \bigg(  \frac{v\sqrt{h^{2}-\widetilde{v}^{2}}}{\widetilde{v}},
\sqrt{h^{2}-v^{2}} \bigg]  \times\left\{  v\right\}  \Bigg)
=\left( F\left(  \frac{\sqrt{h^{2}-v^{2}}}{v}\right)  -
F\left(  \frac{\sqrt {h^{2}-\widetilde{v}^{2}}}{\widetilde{v}}\right)  \right)
\cdot\Lambda\left( \left\{  v\right\}  \right)  .
\end{equation*}
Now
\[
\lim_{\widetilde{v}\searrow v} \Pi \Bigg(  \left(
\frac{v\sqrt{h^{2} -\widetilde{v}^{2}}}{\widetilde{v}},\sqrt{h^{2}-v^{2}} \right]
\times\left\{ v\right\}  \Bigg)
=\Pi\left(  \left\{  \left(  \sqrt{h^{2}-v^{2}},v\right)
\right\}  \right)
\]
and%
\begin{equation*}
\lim_{\widetilde{v}\searrow v}
\Bigg(  F \left(  \frac{\sqrt{h^{2}-v^{2}}}{v} \right)
-F \left(  \frac{\sqrt{h^{2}-\widetilde{v}^{2}}}{\widetilde{v}} \right)  \Bigg)
\cdot \Lambda\left(  \left\{  v\right\}  \right)
=\left(
F\left(  \varphi\left(  v\right)  \right)  -F\left(  \varphi\left(  v\right)
-\right)  \right)  \cdot\Lambda\left(  \left\{  v\right\}  \right)  ,
\end{equation*}
where $\varphi(\cdot)$ is defined in (\ref{phi}). This says that
\[
\Pi\left(  \left\{  \left(  \sqrt{h^{2}-v^{2}},v\right)  \right\}  \right)
=\left(  F\left(  \varphi\left(  v\right)  \right)  -F\left(  \varphi\left(
v\right)  -\right)  \right)  \cdot\Lambda\left(  \left\{  v\right\}  \right)
.
\]
Similarly,
\[
\Pi\left(  \left\{  \left(  -\sqrt{h^{2}-v^{2}},v\right)  \right\}  \right)
=\left(  F\left(  -\varphi\left(  v\right)  \right)  -F\left(  -\varphi\left(
v\right)  -\right)  \right)  \cdot\Lambda\left(  \left\{  v\right\}  \right)
.
\]
We also obtain that with $v=h$,
\[
\Pi\left(  \left\{  \left(  0,h\right)  \right\}  \right)  =\left(  F\left(
0\right)  -F\left(  0-\right)  \right)  \cdot\Lambda\left(  \left\{
h\right\}  \right)  ,
\]
and the proof is complete.
\end{proof}

Let%
\[
B_{h}=\left\{  \left(  u,v \right)  :\sqrt{v^{2}+u^{2}}\leq h,v>0\right\}
\]
and
\[
C_{h}=\left\{  \left(  \sqrt{h^{2}-v^{2}}, v \right)  :0<v\leq h\right\}
\cup\left\{  \left(  -\sqrt{h^{2}-v^{2}}, v \right)  :0<v\leq h\right\}  .
\]

\begin{remark}
\label{remark-cont} Lemma \ref{F-cont} says that when $\Pi\left(
C_{h}\right)  =0$, then $F\left(  \varphi\left(  v\right)  \right)  $ and
$\overline{\Lambda}\left(  v\right)  $ are not discontinuous at the same
points in $\left(  0,h\right]  $; and $F\left(  -\varphi\left(  v\right)
\right)  $ and $\overline{\Lambda}\left(  v\right)  $ are not discontinuous at
the same points in $\left(  0,h\right)  $.
\end{remark}

\begin{lemma} \label{lambda-int}
Suppose (\ref{Gz}) is satisfied and for
every continuity point $h>0$ of $\overline{\Lambda}$, (\ref{Ga}) holds
where $\alpha_{h}<\infty$. Then
\begin{equation}
\int_{0}^{1}z\Lambda(\mathrm{d}z)<\infty. \label{z1}%
\end{equation}
\end{lemma}

\begin{proof}
Let $1\geq h>\gamma>0$, be continuity points of
$\overline{\Lambda}$. By assumptions (\ref{Gz}) and (\ref{Ga})
\[
\alpha_{h}=\lim_{{n^{\prime}}\rightarrow\infty}\int_{0}^{h}vn^{\prime}G\left(
\mathrm{d}a_{n^{\prime}}v\right)  \geq\lim_{{n^{\prime}}\rightarrow\infty}%
\int_{\gamma}^{h}vn^{\prime}G\left(  \mathrm{d}a_{n^{\prime}}v\right)
=\int_{\gamma}^{h}z\Lambda(\mathrm{d}z)\geq0,
\]
which implies that
\[
\infty>\alpha_{h}\geq\lim_{\gamma\searrow0}\int_{\gamma}^{h}z\Lambda
(\mathrm{d}z)=\int_{0}^{h}z\Lambda(\mathrm{d}z)\geq0.
\]
\end{proof}

\begin{remark}
\label{remark-alpha} Applying Lemma \ref{CHT-Kallenberg}, we see that
assumption (\ref{Y-conv}) implies that (\ref{Gz}) and (\ref{Ga}) hold with%
\begin{equation}
\alpha_{h}=b-\int_{h}^{1}z\Lambda(\mathrm{d}z)=b^{h}\text{ and }\alpha
=b-\int_{0}^{1}z\Lambda(\mathrm{d}z)\geq0, \label{alpha}%
\end{equation}
where
\begin{equation}
\alpha=\lim_{h\searrow0}\alpha_{h}\geq0, \label{alpha2}%
\end{equation}
in accordance with the notation in Theorem \ref{2dim-conv}.
This shows that (\ref{alp1}) holds.
\end{remark}

Notice that
\[
\frac{n^{\prime}}{a_{n^{\prime}}}E\left(  Y I \left\{  \sqrt{\left(
XY\right)  ^{2}+Y^{2}}\leq a_{n^{\prime}}h\right\}  \right)  =\int_{B_{h}%
}F\left(  \frac{\mathrm{d}u}{v}\right)  n^{\prime}vG\left(  \mathrm{d}%
a_{n^{\prime}}v\right)
\]
and
\[
\frac{n^{\prime}}{a_{n^{\prime}}}E\left(  XY I \left\{  \sqrt{\left(
XY\right)  ^{2}+Y^{2}}\leq a_{n^{\prime}}h\right\}  \right)  =\int_{B_{h}%
}uF\left(  \frac{\mathrm{d}u}{v}\right)  n^{\prime}G\left(  \mathrm{d}%
a_{n^{\prime}}v\right)  .
\]
Define the functions of $v\in$ $\left(  0,h\right]  $
\[
\phi\left(  v\right)  =\int_{\left[  -\sqrt{h^{2}-v^{2}},\sqrt{h^{2}-v^{2}%
}\right]  }F\left(  \frac{\mathrm{d}u}{v}\right)  =F\left(  \varphi\left(
v\right)  \right)  -F\left(  -\varphi\left(  v\right)  -\right)
\]
and
\[
\psi\left(  v\right)  =\int_{\left[  -\sqrt{h^{2}-v^{2}},\sqrt{h^{2}-v^{2}%
}\right]  }uF\left(  \frac{\mathrm{d}u}{v}\right)  ,
\]
where $\varphi(\cdot)$ is defined in (\ref{phi}). Observe that
\begin{equation}
\phi\left(  v\right)  \nearrow1\text{, as }v\searrow0, \label{zz}%
\end{equation}
and
\begin{equation}
\psi\left(  v\right)  /v\rightarrow EX\text{, as }v\searrow0. \label{ee}%
\end{equation}

Now we can prove the convergence of the truncated expectations.

\begin{proposition}
\label{1st moments} \textit{Assume }(\ref{Gz}), (\ref{Ga}) \textit{and}
$\Pi\left(  C_{h}\right)  =0.$ \textit{Then}
\begin{equation}
\lim_{n^{\prime}\rightarrow\infty}\int_{B_{h}} F\left(  \frac{\mathrm{d} u}%
{v}\right)  n^{\prime}v G\left(  \mathrm{d} a_{n^{\prime}}v\right)
=\alpha+\int_{0}^{h}\phi\left(  v\right)  v \Lambda\left(  \mathrm{d}
v\right)  \label{La}%
\end{equation}
\textit{and}
\begin{equation}
\lim_{n^{\prime}\rightarrow\infty}\int_{B_{h}}u F\left(  \frac{\mathrm{d}
u}{v}\right)  n^{\prime} G\left(  \mathrm{d} a_{n^{\prime}}v\right)  =\alpha E
X+\int_{0}^{h}\psi\left(  v\right)  \Lambda\left(  \mathrm{d} v \right)  .
\label{Lb}%
\end{equation}
\end{proposition}

\begin{proof}
Observe that
\[
\int_{B_{h}}F\left(  \frac{\mathrm{d}u}{v}\right)  n^{\prime}vG\left(
\mathrm{d}a_{n^{\prime}}v\right)  =\int_{0}^{h}\phi\left(  v\right)
vn^{\prime}G\left(  \mathrm{d}a_{n^{\prime}}v\right)
\]
and
\[
\int_{B_{h}}uF\left(  \frac{\mathrm{d}u}{v}\right)  n^{\prime}G\left(
\mathrm{d}a_{n^{\prime}}v\right)  =\int_{0}^{h}\psi\left(  v\right)
n^{\prime}G\left(  \mathrm{d}a_{n^{\prime}}v\right)  .
\]
Choose any $0<\gamma<h$ such that $\gamma$ is a continuity point of
$\overline{\Lambda}$. Notice that since $\Pi\left(  C_{h}\right)  =0$ we can
infer from Remark \ref{remark-cont} that for any such $\gamma$ the functions
of $v$ defined in $\left(  \gamma,h\right]  $ by $\phi\left(  v\right)  v$ and
$\psi\left(  v\right)  $ do not share the same discontinuity points as
$\overline{\Lambda}$. Thus since these functions are also bounded on $\left(
\gamma,h\right]  ,$ assumption (\ref{Gz}) implies as in the argument that
gives (\ref{L3}) that
\begin{equation}
\lim_{n^{\prime}\rightarrow\infty}\int_{\gamma}^{h}\phi\left(  v\right)
vn^{\prime}G\left(  \mathrm{d}a_{n^{\prime}}v\right)  =\int_{\gamma}^{h}%
\phi\left(  v\right)  v\Lambda\left(  \mathrm{d}v\right)  \label{Laa}%
\end{equation}
and
\begin{equation}
\lim_{n^{\prime}\rightarrow\infty}\int_{\gamma}^{h}\psi\left(  v\right)
n^{\prime}G\left(  \mathrm{d}a_{n^{\prime}}v\right)  =\int_{\gamma}^{h}%
\psi\left(  v\right)  \Lambda\left(  \mathrm{d}v\right)  . \label{Lbb}%
\end{equation}
Next, using the monotonicity of $\phi$ we see that
\[
\left\vert \int_{0}^{\gamma}\phi\left(  v\right)  vn^{\prime}G\left(
\mathrm{d}a_{n^{\prime}}v\right)  -\alpha\right\vert
\leq\left\vert 1-\phi\left(  \gamma\right)  \right\vert \int_{0}^{\gamma
}vn^{\prime}G\left(  \mathrm{d}a_{n^{\prime}}v\right)  +\left\vert \alpha
-\int_{0}^{\gamma}vn^{\prime}\mathrm{d}G\left(  a_{n^{\prime}}v\right)
\right\vert .
\]
Therefore, by  (\ref{Ga})
\[
\limsup_{n^{\prime}\rightarrow\infty}\left\vert \int_{0}^{\gamma}\phi\left(
v\right)  vn^{\prime}G\left(  \mathrm{d}a_{n^{\prime}}v\right)  -\alpha
\right\vert \leq\left\vert 1-\phi\left(  \gamma\right)  \right\vert
\alpha_{\gamma}+\left\vert \alpha-\alpha_{\gamma}\right\vert .
\]
Similarly
\begin{equation*}
\limsup_{n^{\prime}\rightarrow\infty}\left\vert \int_{0}^{\gamma}\psi\left(
v\right)  n^{\prime}G\left(  \mathrm{d}a_{n^{\prime}}v\right)  -\alpha
EX\right\vert
\leq\sup_{0<v\leq\gamma} \left\vert EX-v^{-1}\psi\left(
v\right)  \right\vert \alpha_{\gamma}+\left\vert \alpha-\alpha_{\gamma
}\right\vert \left\vert EX \right \vert .
\end{equation*}
As $\gamma\rightarrow0$ the statements follow from the definition of $\alpha$
given in (\ref{alpha2}), (\ref{zz}) and (\ref{ee}).
\end{proof}

Observe that
\[
\frac{n^{\prime}}{a_{n^{\prime}}^{2}} E \left(  Y^{2} I \left\{
\sqrt{\left(  XY\right)  ^{2}+Y^{2}}\leq a_{n^{\prime}}h\right\}  \right)
=\int_{B_{h}} F\left(  \frac{\mathrm{d} u}{v}\right)  n^{\prime}v^{2} G\left(
\mathrm{d} a_{n^{\prime}}v\right)
\]
and
\[
\frac{n^{\prime}}{a_{n^{\prime}}^{2}} E \left(  \left(  XY\right)
^{2}  I \left\{  \sqrt{\left(  XY\right)  ^{2}+Y^{2}}\leq a_{n^{\prime}%
}h\right\}  \right)  =\int_{B_{h}}u^{2} F\left(  \frac{\mathrm{d} u}%
{v}\right)  n^{\prime} G\left(  \mathrm{d} a_{n^{\prime}}v\right)  .
\]

\begin{proposition}
\label{2nd-moments} Assume (\ref{Gz}) and (\ref{Ga}). Then for every $h>0$
such that $\Pi\left(  C_{h}\right)  =0$,
\begin{equation}
\lim_{n^{\prime}\rightarrow\infty}\int_{B_{h}}u^{2} F\left(  \frac{\mathrm{d}
u}{v}\right)  n^{\prime} G\left(  \mathrm{d} a_{n^{\prime}}v\right)
=\int_{B_{h}}u^{2} \Pi\left(  \mathrm{d} u, \mathrm{d} v\right)  , \label{a1}%
\end{equation}%
\begin{equation}
\lim_{n^{\prime}\rightarrow\infty}\int_{B_{h}} F\left(  \frac{\mathrm{d} u}%
{v}\right)  n^{\prime}v^{2} G\left(  \mathrm{d} a_{n^{\prime}}v\right)
=\int_{B_{h}}v^{2} \Pi\left(  \mathrm{d} u, \mathrm{d} v\right)  \label{a2}%
\end{equation}
and
\begin{equation}
\lim_{n^{\prime}\rightarrow\infty}\int_{B_{h}} u v F\left(  \frac{\mathrm{d}
u}{v}\right)  n^{\prime} G\left(  \mathrm{d} a_{n^{\prime}}v\right)
=\int_{B_{h}}u v \Pi\left(  \mathrm{d} u,\mathrm{d} v\right)  . \label{a3}%
\end{equation}
Moreover
\begin{equation}
\lim_{h\searrow0}\limsup_{n^{\prime}\rightarrow\infty}\int_{B_{h}}F\left(
\frac{\mathrm{d} u}{v}\right)  n^{\prime}v^{2} G\left(  \mathrm{d}
a_{n^{\prime}}v\right)  =0 \label{zer1}%
\end{equation}
and
\begin{equation}
\lim_{h\searrow0}\limsup_{n^{\prime}\rightarrow\infty}\int_{B_{h}}u^{2}
F\left(  \frac{\mathrm{d} u}{v}\right)  n^{\prime} G\left(  \mathrm{d}
a_{n^{\prime}}v\right)  =0. \label{zer2}%
\end{equation}
\end{proposition}

\begin{proof}
In the proof of (\ref{zer1}) and (\ref{zer2}) we can
assume without loss of generality that $\Pi\left(  C_{h}\right)  =0$ for all
$h>0$ sufficiently small, since we only need it to be true for a countable
number of $h\searrow0$, and this holds trivially. We see that
\[
\int_{B_{h}} F\left(  \frac{\mathrm{d} u}{v}\right)  n^{\prime}v^{2} G\left(
\mathrm{d} a_{n^{\prime}}v\right)  \leq h\int_{B_{h}} F\left(  \frac{
\mathrm{d} u}{v}\right)  n^{\prime}v G\left(  \mathrm{d} a_{n^{\prime}%
}v\right)
\]
and
\[
\int_{B_{h}}u^{2} F\left(  \frac{\mathrm{d} u}{v}\right)  n^{\prime} G\left(
\mathrm{d} a_{n^{\prime}}v\right)  \leq h\int_{B_{h}}\left\vert u\right\vert
F\left(  \frac{\mathrm{d} u}{v}\right)  n^{\prime} G\left(  \mathrm{d}
a_{n^{\prime}}v\right)  .
\]
Statement (\ref{zer1}) is a consequence of (\ref{La}) and a slight
modification of the argument giving (\ref{Lb}) yields
\[
\lim_{n^{\prime}\rightarrow\infty}\int_{B_{h}}\left\vert u\right\vert F\left(
\frac{\mathrm{d} u}{v}\right)  n^{\prime} G\left(  \mathrm{d} a_{n^{\prime}%
}v\right)  =\alpha E \left\vert X\right\vert +\int_{B_{h}}\left\vert
u\right\vert F\left(  \frac{\mathrm{d} u}{v}\right)  \Lambda\left(  \mathrm{d}
v\right)  \text{,}%
\]
from which (\ref{zer2}) follows.

The proof of the first three limit results now can be carried out the same way
as in the previous proposition.
\end{proof}

Now we are ready to prove Theorem \ref{2dim-conv}.\smallskip

\noindent
\textit{Proof of Theorem \ref{2dim-conv}.}
We have to check the three
conditions in Lemma \ref{CHT-Kallenberg} for the array%
\begin{equation}
\left\{  \left(  X_{i}Y_{i}/a_{n^{\prime}},Y_{i}/a_{n^{\prime}}\right)
\right\}  _{i=1}^{n^{\prime}}. \label{aa2}%
\end{equation}
First of all, assumption (\ref{Y-conv}) permits us to apply Lemma
\ref{CHT-Kallenberg} to the array
\begin{equation}
\left\{  Y_{i}/a_{n^{\prime}}\right\}  _{i=1}^{n^{\prime}}, \label{aa1}%
\end{equation}
to get that (e.i) and (e.ii) in the form (\ref{Gz}) and (\ref{Ga}) are
satisfied for (\ref{aa1}). Thus we can infer from Proposition
\ref{Levy-measure} that (e.i) holds as given in (\ref{L1}) and (\ref{L2}) for
(\ref{aa2}). Next we apply Proposition \ref{1st moments} to see that (e.ii)
holds for (\ref{aa2}) in the form (\ref{La}) and (\ref{Lb}). In particular,
notice that in Proposition \ref{1st moments} we can write
\[
\alpha+\int_{0}^{h}\phi\left(  v\right)  v\Lambda\left(  \mathrm{d}v\right)
=\alpha+\int_{B_{h}}v\Pi\left(  \mathrm{d}u,\mathrm{d}v\right)
\]
and
\[
\alpha EX+\int_{0}^{h}\psi\left(  v\right)  \Lambda\left(  \mathrm{d}v\right)
=\alpha EX+\int_{B_{h}}u\Pi\left(  \mathrm{d}u,\mathrm{d}v\right)  .
\]
Using that
\[
\mathbf{b}^{h}=\mathbf{b}-\int_{h<|(u,v)|\leq1}(u,v)\Pi(\mathrm{d}%
u,\mathrm{d}v),
\]
we get that $\mathbf{b}$ must have the form
\[
\mathbf{b}=\left(
\begin{array}
[c]{c}%
\alpha EX+\int_{0<u^{2}+v^{2}\leq1}u\Pi\left(  \mathrm{d}u,\mathrm{d}v\right)
\\
\alpha+\int_{0<u^{2}+v^{2}\leq1}v\Pi\left(  \mathrm{d}u,\mathrm{d}v\right)
\end{array}
\right)  .
\]

Finally, Proposition \ref{2nd-moments} shows that the covariance matrix $a$
has to be 0, so that (e.iii) holds for (\ref{aa2}) with $a=0$.
\hfill $\Box\medskip$

\noindent
\textit{Proof of Theorem \ref{density}.} The proof will be derived from 
results in Griffin \cite{Griffin}. Note that since both $X$ and $Y$
are independent and non-degenerate, the random vector $\left( XY,Y\right) $ is full, which in
this case means that its distribution is not concentrated on a line. Since $%
Y\in \mathcal{F}_{c}$ there exits a sequence of positive constants $a_{n}$
such that for every subsequence of $\left\{ n\right\} $ there is a further
subsequence $\left\{ n^{\prime }\right\} $ such that $W_{2,n^{\prime
}}/a_{n^{\prime }}$ converges in distribution to a non-degenerate random
variable. Set 
\begin{equation*}
B_{n}=\left( 
\begin{array}{cc}
\frac{1}{a_{n}} & 0 \\ 
0 & \frac{1}{a_{n}}%
\end{array}%
\right) \text{ }.
\end{equation*}%
Clearly, we can now apply Theorem \ref{2dim-conv} to conclude that for every subsequence of $\left\{ n\right\} $ there is a further
subsequence $\left\{ n^{\prime }\right\} $ such that
\begin{equation}
B_{n'}\left( 
\begin{array}{c}
W_{1,n'} \\ 
W_{2,n'}%
\end{array}%
\right) ,  \label{B}
\end{equation}%
converges in distribution along $\left\{ n^{\prime }\right\} $ to a random vector
\begin{equation}
\left( 
\begin{array}{c}
W_{1} \\ 
W_{2}%
\end{array}%
\right) ,  \label{U}
\end{equation}%
which is non-degenerate and full. \textquotedblleft Full\textquotedblright \ follows
by an examination of the structure of the characteristic function of $\left(
W_{1},W_{2}\right) $ given in (\ref{limit-chfunc}).  Thus we see that
condition (C) of Griffin \cite{Griffin} holds.  Next Theorem 4.5 of Griffin 
\cite{Griffin} says the conditions (A) and (C) of \cite{Griffin} are
equivalent. Now since condition (A) of \cite{Griffin} is satisfied, we can use the
proof of Griffin's Theorem 4.1 to show that there exist sequences of linear transformations 
$A_{n}: \mathbb{R}^2 \rightarrow \mathbb{R}^2$ and vectors $\delta _{n} \in \mathbb{R}^2$ such that 
\begin{equation*}
A_{n}\left\{ \left( 
\begin{array}{c}
W_{1,n} \\ 
W_{2,n}%
\end{array}%
\right) -\delta _{n}\right\}
\end{equation*}%
is stochastically compact and all of its subsequential distributional limit
random vectors, say, 
\begin{equation}
\left( 
\begin{array}{c}
W_{1}^{\prime } \\ 
W_{2}^{\prime }%
\end{array}%
\right)   \label{U'}
\end{equation}%
are non-degenerate and full. Moreover, Griffin proves that any such random vector (%
\ref{U'}) has a $C^{\infty }$ density. This fact combined with an argument
based on the convergence of types theorem implies that each subsequential
limit random vector (\ref{U}) has a $C^{\infty }$ density, say $f\left(
u,v\right) $. (See the convergence of types theorem given in Theorem 2.3.17
on page 35 in \cite{MS}.) Thus since every subsequential limit (\ref{B})
is full with density $f\left( u,v\right) $, the distributional limit $T$ of the
corresponding self-normalized sum (\ref{TT}) has density 
\begin{equation*}
f_{T}\left( t\right) =\int_{0}^{\infty }vf\left( tv,v\right) \mathrm{d}v.
\end{equation*}%
\hspace{1pt} \hfill $\Box \medskip $

\noindent
\textit{Proof of Proposition \ref{alpha0}.}
It can be inferred from classical theory (or from the proof
of Theorem \ref{2dim-conv}) that every subsequential
limit law $W$ of $a_{n}^{-1}\sum_{i=1}^{n}Y_{i}$ has the
$\mathrm{id}(\alpha,\Lambda)$ distribution with characteristic function
\begin{equation*}
Ee^{\mathrm{i}uW}= \exp\left\{  \mathrm{i}u \alpha+
\int_{0}^{\infty}\left( e^{\mathrm{i}ux} - 1 \right)  \Lambda(\mathrm{d}x)\right\},
\end{equation*}
where $\Lambda$ satisfies (\ref{ff}), and $\alpha \geq 0$.
Clearly $W \overset{\mathrm{D}}{=}\alpha+V$ and the L\'{e}vy process
associated with $W$ is $\alpha t+V_{t}$, $t\geq0$, where
\[
Ee^{\mathrm{i}uV_{t}}=\exp\left\{  t\int_{0}^{\infty}\left(  e^{\mathrm{i}%
uy}-1\right)  \Lambda\left(  \mathrm{d}y\right)  \right\}  .
\]
By an application of Corollary 1 of Maller and Mason \cite{MM2} this implies that
the process $\alpha t+V_{t}$, $t\geq0$, is both in the \textit{centered Feller
class} at zero and at infinity.
Using the notation of \cite{MM2} and \cite{MM3} we have
\[
\nu(x)=\gamma_{\alpha}+\int_{1}^{x}y\Lambda(\mathrm{d}y)=\alpha
+\int_{0}^{x}y\Lambda(\mathrm{d}y),
\]
where $\gamma_\alpha = \alpha+ \int_0^1 y\Lambda(\mathrm{d}y)$.
We get by Theorem 2.3 in
Maller and Mason \cite{MM3} (equation (2.11)) that for some $C>0$ for all $x>0$
small enough
\[
x \left( \alpha + \int_0^x y \Lambda(\mathrm{d} y) \right)
\leq C \int_0^x y^2 \Lambda( \mathrm{d} y),
\]
and thus
\[
\alpha + \int_0^x y \Lambda(\mathrm{d} y) \leq \frac{C}{x} \int_0^x y^2 \Lambda( \mathrm{d} y) 
\leq C \int_0^x y \Lambda (\mathrm{d} y),
\]
and the upper bound tends to 0, as $x \searrow 0$.
Since $\alpha\geq0,$ this can only happen if $\alpha=0$.
\hspace*{1pt} \hfill $\Box\medskip$

\noindent
\textit{Proof of Theorem \ref{converse}.} Choose any $0<\varepsilon
<1$, then on the set $A_{n}\left(  \varepsilon\right)  $ for any $k>1$, by the
conditional version of Chebyshev's inequality
\begin{equation} \label{cc}
\begin{split}
& P\left\{  \left\vert \frac{\sum_{i\not =m(n)}X_{i}Y_{i}}{\sum_{i=1}^{n}Y_{i}%
}\right\vert >\varepsilon\sqrt{k}E|X|\text{ }\bigg|A_{n}\left(  \varepsilon
\right)  \right\} \\
\leq \, & E\left(  \left\vert \frac{\sum_{i\not =m(n)}X_{i}Y_{i}}{\sum_{i=1}%
^{n}Y_{i}}\right\vert |A_{n}\left(  \varepsilon\right)  \right)  /\left(
\varepsilon\sqrt{k}E|X|\right)  \leq k^{-1/2}.
\end{split}
\end{equation}
Let $\varepsilon=1/k$ and set
\[
B_{k,n}=\left\{  \left\vert \frac{\sum_{i\not =m(n)}X_{i}Y_{i}}{\sum_{i=1}%
^{n}Y_{i}}\right\vert \leq k^{-1/2}E|X|\right\}  .
\]
We get by (\ref{cc}) that
\[
P\left\{  B_{k,n}|A_{n}\left(  k^{-1}\right)  \right\}  \geq1-k^{-1/2}\text{.}%
\]
On the set $A_{n}\left(  k^{-1}\right)  \cap B_{k,n}$ we have
\[
\Delta_{n}\leq\left\vert X_{m\left(  n\right)  }\right\vert k^{-1}%
+k^{-1/2}E|X|.
\]
Now for any $0<\eta<1$ there exists a $K_{\eta}>0$ such that $P\left\{
\vert X_{m (n)  } \vert \leq K_{\eta}\right\}
\geq1-\eta$. Observe that
\begin{align*}
& P \left\{ \Delta_{n}\leq  K_{\eta}  k^{-1}+k^{-1/2}E|X| \right\} \\
& \geq P \left\{ \Delta_{n}\leq \big \vert X_{m (n)} \big\vert
k^{-1}+k^{-1/2}E|X|,\big\vert X_{m (n) } \big\vert \leq
K_{\eta} \right\} \\
& \geq  P \left\{ \Delta_{n}\leq \big\vert X_{m (n) } \big\vert
k^{-1}+k^{-1/2}E|X| \right\}-P \left\{  \big\vert X_{m (n) } \big\vert
>K_{\eta} \right\}  ,
\end{align*}
which is
\[
\geq P\left\{  A_{n} \big(  k^{-1}\big)  \cap B_{k,n}\right\}
\mathbf{-}\eta=P\{A_{n} \big(  k^{-1} \big) \} P\left\{  B_{k,n}|A_{n}
\big( k^{-1} \big)  \right\}  - \eta\text{.}%
\]
Therefore we have with $\varepsilon_{k} (  \eta )  :=K_{\eta}%
k^{-1}+k^{-1/2}E|X|$,
\[
P\left\{  \Delta_{n}\leq\varepsilon_{k}\left(  \eta\right)  \right\}  \geq
P\{A_{n} \big(  k^{-1} \big)  \}\left(  1-k^{-1/2}\right)  \mathbf{-}%
\eta\mathbf{.}%
\]
Notice that for each fixed $\eta>0$ and $\delta^{\prime}<\delta$ for all large
enough $k$ and large enough $n^{\prime}$ along the subsequence $\{n^{\prime
}\}$ as in (\ref{dd})
\[
P \left\{A_{n^{\prime}}\big(  k^{-1}\big)  \right\}\left(  1-k^{-1/2}\right)
\mathbf{-}\eta\geq\delta^{\prime}-\eta.
\]
Clearly we can choose $\delta^{\prime}<\delta$ sufficiently close to $\delta$
and $\eta>0$ small enough so that $\delta^{\prime}-\eta$ is as close to
$\delta$ as desired: Since for each fixed $\eta>0$, $\varepsilon_{k}\left(
\eta\right)  \rightarrow0$, as $k\rightarrow\infty$, we see that statement
(\ref{BigD}) holds along the subsequence $\{n^{\prime}\}$ as in (\ref{dd}).
\hfill $\Box\medskip$

\noindent
\textit{Proof of Theorem \ref{exp}.}
First we introduce some
notation. Set for any $C>0$ and random variable $Z$, $Z^{C}=Z I\left\{
\left\vert Z\right\vert \leq C\right\}  $ and $\overline{Z}^{C}=Z-Z^{C}$.
Define the random variables for $n\geq1$%
\[
S_{n}=\frac{\sum_{i=1}^{n}\left(  X_{i}-EX\right)  Y_{i}}{\sum_{i=1}^{n}Y_{i}%
},\text{ }S_{n}^{C}=\frac{\sum_{i=1}^{n}\left(  X_{i}^{C}-EX^{C}\right)
Y_{i}}{\sum_{i=1}^{n}Y_{i}}, \ \overline{S}_{n}^{C}=S_{n}-S_{n}^{C},
\]%
\[
\text{ }N_{n}^{C}=\frac{\sum_{i=1}^{n}\left(  X_{i}^{C}-EX^{C}\right)  Y_{i}%
}{\sqrt{\sum_{i=1}^{n}Y_{i}^{2}}}\text{ and }R_{n}=\frac{\sqrt{\sum_{i=1}%
^{n}Y_{i}^{2}}}{\sum_{i=1}^{n}Y_{i}}.
\]
As we noted before by the results of Griffin \cite{Grif} our assumption that
(\ref{grif}) does not hold is equivalent to
\begin{equation}
R_{n}^{-1}\neq O_{P}\left(  1\right)  , \label{R}%
\end{equation}
so there exist a $\delta>0$ and a subsequence $\left\{  n_{k}\right\}  $ of
$\left\{  n\right\}  $ such that $n_{k}\rightarrow\infty$ and
\begin{equation}
\lim_{\eta\searrow0}\liminf_{k\rightarrow\infty}P\left\{  R_{n_{k}}\leq
\eta\right\}  =\delta. \label{delta}%
\end{equation}
Now for any $\eta>0$, $C>0$ and $K>0$%
\begin{align*}
& P\left\{  \left\vert S_{n_{k}}\right\vert \leq\eta K\sqrt{Var\left(
X^{C}\right)  }+KE\left\vert \overline{X}^{C}\right\vert \right\} \\
& \geq P\left\{  \left\vert S_{n_{k}}^{C}\right\vert \leq\eta K\sqrt{Var\left(
X^{C}\right)  },\left\vert \overline{S}_{n_{k}}^{C}\right\vert \leq
KE\left\vert \overline{X}^{C}\right\vert \right\} \\
& \geq P\left\{  \left\vert S_{n_{k}}^{C}\right\vert \leq\eta K\sqrt{Var\left(
X^{C}\right)  }\right\}  -P\left\{  \left\vert \overline{S}_{n_{k}}%
^{C}\right\vert >KE\left\vert \overline{X}^{C}\right\vert \right\}  .
\end{align*}
Note that by Markov's inequality
\begin{equation}
P\left\{  \left\vert \overline{S}_{n_{k}}^{C}\right\vert >KE\left\vert
\overline{X}^{C}\right\vert \right\}  \leq E\left\vert \overline{X}%
^{C}-E\overline{X}^{C}\right\vert /\left(  KE\left\vert \overline{X}%
^{C}\right\vert \right)  \leq2/K. \label{e1}%
\end{equation}
Write $S_{n_{k}}^{C}=N_{n_{k}}^{C}R_{n_{k}}.$ Now
\[
P\left\{  \left\vert S_{n_{k}}^{C}\right\vert \leq\eta K\sqrt{Var\left(
X^{C}\right)  }\right\}  \geq P\left\{  R_{n_{k}}\leq\eta,\left\vert N_{n_{k}%
}^{C}\right\vert \leq K\sqrt{Var\left(  X^{C}\right)  }\right\}
\]%
\[
\geq P\left\{  R_{n_{k}}\leq\eta\right\}  -P\left\{  \left\vert N_{n_{k}}%
^{C}\right\vert >K\sqrt{Var\left(  X^{C}\right)  }\right\}  ,
\]
which by Chebyshev's inequality is%
\[
\geq P\left\{  R_{n_{k}}\leq\eta\right\}  -1/K^{2}.
\]
Thus for each $\eta>0$, $C>0$ and $K>0$%
\[
P\left\{  \left\vert S_{n_{k}}\right\vert \leq\eta K\sqrt{Var\left(
X^{C}\right)  }+KE\left\vert \overline{X}^{C}\right\vert \right\}
\geq P\left\{  R_{n_{k}}\leq\eta\right\}  -1/K^{2}-2/K.
\]
Next note that for large enough $K$
\[
1/K^{2}+2/K<\delta/4.
\]
Also for any $\varepsilon>0$, for all large enough $C>0$%
\[
KE\left\vert \overline{X}^{C}\right\vert \leq\varepsilon/2
\]
and given $C>0$ and $K>0$ for a small enough $\eta>0$,%
\[
\eta K\sqrt{Var\left(  X^{C}\right)  }\leq\varepsilon/2.
\]
This gives
\begin{equation*}
\begin{split}
P\left\{  \left\vert S_{n_{k}}\right\vert \leq\varepsilon\right\}
&  \geq P \left\{  \left\vert S_{n_{k}}\right\vert \leq\eta K\sqrt{Var\left(
X^{C}\right)  }+KE\left\vert \overline{X}^{C}\right\vert \right\} \\
& \geq P\left\{  R_{n_{k}}\leq\eta\right\}  -\delta/4.
\end{split}
\end{equation*}
Thus by (\ref{delta}) for all large enough $k$%
\[
P\left\{  \left\vert S_{n_{k}}\right\vert \leq\varepsilon\right\}  \geq
\delta/4,
\]
which since $\varepsilon>0$ is independent of $\delta$, implies that
\begin{equation}
\lim_{\varepsilon\searrow0}\liminf_{k\rightarrow\infty}P\left\{  \left\vert
S_{n_{k}}\right\vert \leq\varepsilon\right\}  \geq\delta/4. \label{delta1}%
\end{equation}
To complete the proof, notice that
\[
S_{n_{k}}=O_{P}\left(  1\right)  ,
\]
which implies by tightness that there exists a subsequence $\left\{
n^{\prime}\right\}  $ of $\left\{  n_{k}\right\}  $ and a random variable $S$
\[
S_{ n^{\prime}}\overset{\mathrm{D}}{\longrightarrow}S\text{, }%
\]
which by (\ref{delta1}) satisfies $P\left\{  S =0\right\}  \geq\delta/4.$
\hfill $\Box\medskip$

We are now ready to prove Theorem \ref{continuous}. \smallskip

\noindent
\textit{Proof of Theorem \ref{continuous}. }
Theorem \ref{density}
implies that if $Y\in\mathcal{F}_{c}$ then every subsequential law of $T_{n}$
has a Lebesgue density.\smallskip

Now suppose that $Y\notin\mathcal{F}_{c}$. Applying a characterization of
Maller \cite{M} we know that $Y$ is in the centered Feller class if and only
if
\begin{equation*}
\limsup_{x\rightarrow\infty}\frac{x^{2}P\{Y>x\}+xE\left(  YI(Y\leq x)\right)
}{EY^{2}I(Y\leq x)}<\infty. 
\end{equation*}
Thus if $Y\notin\mathcal{F}_{c}$%
\begin{equation*}
\limsup_{x\rightarrow\infty}\frac{x^{2}P\{Y>x\}}{EY^{2}I(Y\leq x)}%
=\infty\text{ or }\limsup_{x\rightarrow\infty}\frac{xE\left(  YI(Y\leq
x)\right)  }{EY^{2}I(Y\leq x)}=\infty. 
\end{equation*}
Note that if $Y\notin\mathcal{F}_{c}$ and (\ref{grif}) does not hold we can
apply Theorem \ref{exp} to show that for some subsequence $\left\{  n^{\prime
}\right\}  $, $T_{n^{\prime}}\overset{\mathrm{D}}{\longrightarrow}T,$ where
$P\left\{  T=EX\right\}  >0.$ Next, if $Y\notin\mathcal{F}_{c}$ but
(\ref{grif}) is satisfied then (\ref{not-feller-class}) must hold too. Thus by
the fact that (\ref{not-feller-class}) and (\ref{grif}) imply that (\ref{dd})
holds, we can apply Corollary \ref{converse1} to find an $X$ and $x_{0}$ so
that along a subsequence $\left\{  n^{\prime}\right\}  $, $T_{n^{\prime}%
}\overset{\mathrm{D}}{\longrightarrow}T,$ where $P\left\{  T=X=x_{0}\right\}>0$.
\hfill $\Box$

\section{Appendix}

\begin{proposition}
\label{Feller-class} Let $X$ and $Y$ non-degenerate independent random
variables. If $X$ and $Y$ are in the Feller class, then so is $XY$.
\end{proposition}

\begin{proof}
Let denote $F$ and $G$ the distribution functions of
$\left\vert X\right\vert $ and $\left\vert Y\right\vert $ respectively. Since
$Y\in\mathcal{F}$
\begin{equation}
\limsup_{x\rightarrow\infty}\frac{x^{2}P\{\left\vert Y\right\vert >x\}}%
{EY^{2}I(\left\vert Y\right\vert \leq x)}<\infty, \label{feller-class}%
\end{equation}
which means that there is a $K>0$ and $x_{0}>0$, such that for all $x\geq
x_{0}$
\[
\frac{x^{2}P\{\left\vert Y\right\vert >x\}}{EY^{2}I(\left\vert Y\right\vert
\leq x)}<K.
\]

We show that (\ref{feller-class}) holds for $XY$. We have that
\begin{align*}
EX^{2}Y^{2}I(\left\vert XY\right\vert \leq t)  &  =\int\hspace{-5pt}%
\int_{xy\leq t}x^{2}y^{2}F(\mathrm{d}x)G(\mathrm{d}y)\\
&  =\int_{0}^{\infty}x^{2}F(\mathrm{d}x)\int_{0}^{t/x}y^{2}G(\mathrm{d}y)\\
&  \geq\int_{0}^{t/x_{0}}x^{2}F(\mathrm{d}x)\int_{0}^{t/x}y^{2}G(\mathrm{d}y).
\end{align*}
Since $x\leq t/x_{0}$, $t/x\geq x_{0}$, so we can use the estimate above to
obtain
\begin{align*}
&  \geq\int_{0}^{t/x_{0}}x^{2}\frac{1}{K}\frac{t^{2}}{x^{2}}P\{\left\vert
Y\right\vert >t/x\}F(\mathrm{d}x)\\
&  =\frac{t^{2}}{K}\int_{0}^{t/x_{0}}P\{\left\vert Y\right\vert
>t/x\}F(\mathrm{d}x)\\
&  =\frac{t^{2}}{K}P\{\left\vert XY\right\vert >t,\,\left\vert X\right\vert
\leq t/x_{0}\}.
\end{align*}
Now, using that
\begin{align*}
P\{\left\vert XY\right\vert >t,\,\left\vert X\right\vert \leq t/x_{0}\}  &
=P\{\left\vert XY\right\vert >t\}-P\{\left\vert XY\right\vert >t,\,\left\vert
X\right\vert >t/x_{0}\}\\
&  \geq P\{\left\vert XY\right\vert >t\}-P\{\left\vert X\right\vert
>t/x_{0}\},
\end{align*}
we obtain
\[
EX^{2}Y^{2}I(\left\vert XY\right\vert \leq t)\geq\frac{t^{2}}{K}\left(
P\{\left\vert XY\right\vert >t\}-P\{\left\vert X\right\vert >t/x_{0}\}\right)
,
\]
i.e.
\[
\frac{t^{2}P\{\left\vert XY\right\vert >t\}}{EX^{2}Y^{2}I(\left\vert
XY\right\vert \leq t)}\leq K+\frac{t^{2}P\{\left\vert X\right\vert >t/x_{0}%
\}}{EX^{2}Y^{2}I(\left\vert XY\right\vert \leq t)},
\]
so we only have to show that the $\limsup$ of the last term is finite. In
order to do this notice that
\[
EX^{2}Y^{2}I(\left\vert XY\right\vert \leq t)\geq EX^{2}I(\left\vert
X\right\vert \leq t/x_{0})EY^{2}I(\left\vert Y\right\vert \leq x_{0}).
\]
From this we have
\[
\frac{t^{2}P\{\left\vert X\right\vert >t/x_{0}\}}{EX^{2}Y^{2}I(\left\vert
XY\right\vert \leq t)}\leq\frac{x_{0}^{2}}{EY^{2}I(\left\vert Y\right\vert
\leq x_{0})}\frac{(t/x_{0})^{2}P\{\left\vert X\right\vert >t/x_{0}\}}%
{EX^{2}I(\left\vert X\right\vert \leq t/x_{0})},
\]
and the finiteness of the $\limsup$ of the last factor is exactly the
condition $X\in\mathcal{F}$. The proof is finished.
\end{proof}




\end{document}